\documentclass[11pt]{article}
\usepackage[utf8]{inputenc}
\usepackage{lmodern}
\usepackage{subfiles}
\usepackage{enumitem}
	\setenumerate{ref=(\roman*), label={\normalfont(\roman*)}, itemsep=0em} 
        
\usepackage{amsfonts}
\usepackage{amsthm}
\usepackage{amsmath}
\usepackage{amssymb}
\usepackage{amscd}
\usepackage{mathrsfs}
\usepackage{mathtools}
\usepackage{bm}
\usepackage{esint}

\usepackage[margin=3cm]{geometry}
\usepackage{setspace}
\usepackage{indentfirst}
\usepackage{graphicx}
\usepackage{graphics}
\usepackage{lscape}
\usepackage{pgf,tikz}
\usepackage{tikz-cd}
\usepackage{color}
\usepackage{pict2e}
\usepackage{epic}
\usepackage{epstopdf}
\usepackage{titlesec, titlefoot}
	\titleformat{\section}[block]{\Large\bfseries\filcenter}{\thesection}{1em}{}
\usepackage{commath}
\usepackage{float}
\usepackage{caption}
\usepackage{etoolbox}
\usepackage[affil-it]{authblk}
\usepackage{combelow}

\usepackage[hidelinks, bookmarksdepth=3]{hyperref}
\hypersetup{bookmarksopen=true} 
\usepackage{hypcap}

\graphicspath{{./Pictures/}}
\allowdisplaybreaks

\usepackage{listings}
\usepackage{fancybox}

\theoremstyle{plain}
\newtheorem{bigthm}{Theorem}

\newtheorem{bigcor}[bigthm]{Corollary}

\renewcommand*\thesection{\arabic{section}}
\numberwithin{equation}{section} 

\theoremstyle{plain}
\newtheorem{thm}{Theorem}
\newtheorem{lemma}[thm]{Lemma}
\newtheorem{prop}[thm]{Proposition}
\newtheorem{cor}[thm]{Corollary}
\numberwithin{thm}{section} 

\theoremstyle{definition}
\newtheorem{ndef}[thm]{Definition}
\newtheorem{ex}[thm]{Example}
\newtheorem{question}[thm]{Question}

\newtheorem{remark}[thm]{Remark}
\newtheorem{conj}[thm]{Conjecture}

\newcommand{\thistheoremname}{}
\newtheorem{genericthm}[thm]{\thistheoremname}
\newenvironment{para}[1]
  {\renewcommand{\thistheoremname}{#1}%
   \begin{genericthm}}
  {\end{genericthm}}

\newcommand{\thistheoremnames}{}
\newtheorem*{genericthms}{\thistheoremnames}
\newenvironment{para*}[1]
  {\renewcommand{\thistheoremnames}{#1}%
   \begin{genericthms}}
  {\end{genericthms}}

\expandafter\let\expandafter\oldproof\csname\string\proof\endcsname
\let\oldendproof\endproof
\renewenvironment{proof}[1][\proofname]{%
  \oldproof[\upshape \bfseries #1]%
}{\oldendproof}

\makeatletter
\def\@makechapterhead#1{%
  \vspace*{50\p@}%
  {\parindent \z@ \raggedright \normalfont
    \interlinepenalty\@M
    \Huge\bfseries  \thechapter.\quad #1\par\nobreak
    \vskip 40\p@
  }}
\makeatother

\usepackage{tocloft}
\cftsetindents{section}{0em}{2em}
\cftsetindents{subsection}{0em}{2em}
\def \a{\alpha}
\def \R {\mathbb{R}}
\def \Z {\mathbb{Z}}
\def \C{\mathbb{C}}

\def \N{\mathbb{N}}
\def \D{\textup{D}}
\def \J{\textup{J}}

\def \S{\mathbb{S}}
\def \e{\varepsilon}
\def \d{\,\textup{d}}

\def \exc{\backslash}
\def \p{\partial}

\def \mc{\mathcal}
\def \mb{\mathbb}

\def \supp{\textup{supp}\,}

\def \hs{\hspace{0.5cm}}

\def \tp{\textup}

\begin{document}

\title{\textbf{Energy minimisers with prescribed Jacobian}}

\author[1]{{\Large Andr\'e Guerra}}
\author[1]{{\Large Lukas Koch}}
\author[2]{{\Large Sauli Lindberg}}

\affil[1]{\small University of Oxford, Andrew Wiles Building Woodstock Rd, Oxford OX2 6GG, United Kingdom 
\protect \\
  {\tt{\{andre.guerra, lukas.koch\}@maths.ox.ac.uk}}
  \vspace{1em} \ }

\affil[2]{\small 
Aalto University, Department of Mathematics and Systems Analysis, P.O. Box 11100, FI-00076 Aalto, Finland 
\protect\\
  {\tt{sauli.lindberg@aalto.fi}} \ }

\date{}

\maketitle

\begin{abstract}
We study the symmetry and uniqueness of maps which minimise the $np$-Dirichlet energy, under the constraint that their Jacobian is a given radially symmetric function $f$. We find a condition on $f$ which ensures that the minimisers are symmetric and unique. In the absence of this condition we construct an explicit $f$ for which there are uncountably many distinct energy minimisers, none of which is symmetric. Even if we prescribe the maps to be the identity on the boundary of a ball we show that the minimisers need not be symmetric. This gives a negative answer to a question of H\'{e}lein (Ann.\ Inst.\ H.\ Poincaré Anal.\ Non Linéaire 11 (1994), no.\ 3, 275–296).
\end{abstract}

\unmarkedfntext{
\hspace{-0.8cm}
\emph{Formal acknowledgments.} A.G. and L.K. were supported by the EPSRC [EP/L015811/1]. S.L. was supported by the AtMath Collaboration at the University of Helsinki and the ERC grant 834728-QUAMAP.
}

\section{Introduction}

Given a bounded domain $\Omega\subset \R^n$, a typical problem in nonlinear elastostatics is to
\begin{equation}
\tp{minimise } \mathscr W[u]\equiv \int_\Omega W(x,\D u)\d x, \quad \tp{ among all } u\in u_0+W_0^{1,\infty}(\Omega,\R^n).
\label{eq:hyperelasticity}
\end{equation}
In this context, it is standard to assume that $u_0\colon \overline\Omega\to \overline O$ is an orientation-preserving diffeomorphism and that the stored-energy function $W\colon \overline \Omega\times \tp{GL}^+(n)\to \R$ is continuous \cite{Ball1977,Ball1981a}.

In order for the Direct Method to be applicable $\mathscr W$ needs to sequentially weakly lower semicontinuous and, more importantly for our purposes here, coercive \cite{Chen2017}. However, one is sometimes led to consider non-coercive energy functions and this is the case, for instance, when $W$ depends on $\D u$ only through $\J u\equiv \det\D u$, see e.g.\ \cite{Chipot1992, Fonseca1988, Fonseca1992} for examples in the study of elastic crystals and \cite{Dacorogna1981} for a different example. In this case, a solution of \eqref{eq:hyperelasticity} is found, at least formally, by solving
\begin{equation}
\label{eq:dirichlethyperelasticity}
\begin{cases}
\J u = f & \tp{in }\Omega,\\
u=u_0 & \tp{on } \p\Omega,
\end{cases}
\qquad
\tp{where }f \tp{ is such that } W(x,f(x))\equiv \min_{\xi>0} W(x,\xi).
\end{equation}
Although there is a good existence theory for \eqref{eq:dirichlethyperelasticity} whenever there is some $\a>0$ such that $u_0$ is a $C^{1,\a}$-diffeomorphism and $f\in C^{0,\a}$  \cite{Csato2012,Dacorogna1990, Moser1965, Ye1994}, little is known when $f$ is just an $L^p$ function. We refer the reader to  \cite{Fischer2019,Hytonen2018,Koumatos2015,Lindberg2017,Riviere1996}, as well as our
our recent works \cite{GuerraKochLindberg2020b, GuerraKochLindberg2020} for results in this direction;  counter-examples in the $p=\infty$ case were obtained in  \cite{Burago1998,McMullen1998}. 

In light of the above discussion, we are led to study the prescribed Jacobian equation
\begin{equation}
\label{eq:jac}
\J u = f\quad \tp{in } \Omega.
\end{equation}
Apart from the nonlinear character of the Jacobian, the main obstacle in studying existence and regularity of solutions of \eqref{eq:jac} is the \textit{underdetermined} nature of the equation. A natural way to circumvent this issue is to look for solutions which minimise an appropriate energy function: it may be that such solutions are better behaved and unique. Since we are interested in determining the optimal Sobolev regularity of solutions to \eqref{eq:jac}, the natural energy to consider is the $np$-Dirichlet energy:

\begin{ndef}\label{def:energy}
For each $f\in L^1_\tp{loc}(\Omega)$, we define the \textit{$np$-energy of $f$} as 
$$\mc E_{np}(f,\Omega)\equiv \inf\left\{\int_{\Omega}|\D v|^{np}\d x : v\in {\dot W}^{1,np}(\Omega,\R^n) \tp{ satisfies } \J v = f\tp{ a.e. in } \Omega\right\}.$$
Given $f\in L^1_\tp{loc}(\Omega)$, we say that $u\in {\dot
  W}^{1,np}(\Omega, \R^n)$ is a $np$-\textit{energy minimiser for $f$} if 
$$\int_{\Omega} |\D u|^{np}\d x = \mc E_{np}(f,\Omega) \qquad \tp{ and } \qquad \J u = f \tp{ a.e. in } \Omega.$$
\end{ndef}

We contrast this definition with that of the theory of Optimal Transport: there, a typical choice of functional to minimise is
\begin{equation*}
\label{eq:quadratic}
\int_{\Omega}\frac{|u(x)-x|^2}{2}\,f(x) \d x,
\end{equation*}
see \cite{Brenier1991, DePhilippis2014}.
While this energy does single out a unique solution of \eqref{eq:jac}, 
in general there are other solutions of \eqref{eq:jac} with better regularity, see \cite[page 293]{Ye1994} as well as \cite[page 2]{GuerraKochLindberg2020}. The same phenomenon was observed by \textsc{Bourgain} and \textsc{Brezis} in \cite{Bourgain2002} for the divergence equation: in general, underdetermined equations often admit solutions which have a surprising amount of regularity.

Establishing regularity of energy minimisers for \eqref{eq:jac} is a difficult task, even in the incompressible case $f=1$. There is an extensive literature on the topic, and we refer the reader to \cite{Bauman1992, Chaudhuri2009, Evans1999, Karakhanyan2012, Karakhanyan2014}, as well as the references therein, for further information.  

In this paper we are particularly concerned with the following question:
\begin{center}
is energy minimisation an effective selection criterion for \eqref{eq:jac}?
\end{center}

We focus on the case $n=2$. In general, whenever $f$ is bounded away from zero, solutions of \eqref{eq:jac} with $W^{1,n}$-regularity have integrable distortion. By the fundamental results of \textsc{Iwaniec} and \textsc{\v Sverák} \cite{Iwaniec1993a}, maps of integrable distortion between planar domains are \textit{open} and \textit{discrete}.  If $n>2$ then, in order for these properties to hold without further assumptions, the maps need to have $L^p$-integrable distortion for some $p>n-1$, see \cite{Hencl2014a,Iwaniec2001} and the references therein. Nevertheless most of the results in this paper permit straightforward extensions to the higher dimensional setting.

Besides the Dirichlet problem, as in \eqref{eq:dirichlethyperelasticity} and \cite{GuerraKochLindberg2020}, it is also natural to study \eqref{eq:jac} in the case $\Omega=\R^2$. These two situations differ in some important aspects and thus will be considered in separate subsections.

\subsection{Energy minimisers for the Dirichlet problem}

In this subsection we consider the Dirichlet problem for \eqref{eq:jac}: given a domain $\Omega\subset \R^2$,
\begin{equation}
\label{eq:jacdirichlet}
\begin{cases}
\J u = f & \tp{a.e.\ in }\Omega,\\
u=\tp{id} & \tp{on } \p\Omega.
\end{cases}
\end{equation}
We assume that $f$ is compatible with the boundary condition, as well as uniformly positive:
\begin{equation}
\label{eq:conditiondata}
\fint_\Omega f \d x = 1, \qquad \inf_\Omega f\geq c>0, 
\end{equation}
where $c$ is some fixed constant. 

Whenever $\Omega=B$ is a ball centred at the origin and $f$ is a radially symmetric function, i.e.\ $f=f(|z|)$, which satisfies \eqref{eq:conditiondata}, there is a unique radial stretching, denoted $\phi_1$, which solves \eqref{eq:jacdirichlet}:
$$\phi_1(z)\equiv \rho(|z|) \frac{z}{|z|}, \qquad \tp{where }
\rho(r) \equiv \sqrt{\int_0^r 2 s f(s) \d s}.$$
Here the subscript `1' denotes the topological degree of the map, see  Definition \ref{def:radialstretching} for more general solutions. We are thus led to the following very natural problem, c.f.\ \cite[Question 9]{Ye1994}:

\begin{question}[Hélein]\label{question:Helein}
Let $f$ be radially symmetric. Is the unique radial stretching $\phi_1$ solving \eqref{eq:jac} a Dirichlet energy minimiser\footnote{In higher dimensions, one should consider the $n$-harmonic energy, as done in Definition \ref{def:energy}. We believe the original question in \cite{Ye1994} has a misprint.} for $f$ in the class $\tp{id}+W^{1,2}_0(B,B)$?\end{question}

In Theorem \ref{bigthm:helein}, we give a negative answer to Question \ref{question:Helein} and identify a class of data for which uniqueness holds. As in \cite{GuerraKochLindberg2020}, we introduce the complete metric spaces 
$$X_p(B)\equiv \begin{cases}
\{f\in L^p(B):\eqref{eq:conditiondata} \tp{ holds}\} & \tp{if }p>1,\\
\{f\in L\log L(B):\eqref{eq:conditiondata} \tp{ holds}\} & \tp{if }p=1.
\end{cases}$$
\begin{bigthm}\label{bigthm:helein}
Energy minimisers are symmetric for some data but not for other. Precisely:
\begin{enumerate}
\item\label{it:averagecond} Take $1\leq p<\infty$ and let $f\in X_p(B)$ be radially symmetric. If, for a.e.\ $r\in (0,R)$,
\begin{equation}
\label{eq:averagecondition}
f(r)\leq  \fint_{B_r(0)} f  \d x,
\end{equation}
then the unique radial stretching $\phi_1$ solving \eqref{eq:jac} is a $2p$-energy minimiser for $f$ and
\begin{equation}
\Vert \D \phi_1\Vert_{L^{2p}(B)}^2 \lesssim \Vert f\Vert_{X_p(B)},
\label{eq:estimateoverdomains}
\end{equation}
Moreover, $\phi_1$ is the unique $2p$-energy minimiser for $f$ in the class $\tp{id}+W^{1,2p}_0(B,B)$.

\item\label{it:counterex} Fix $p_0\in [1,\infty)$. There is $f\in L^\infty(B)$ satisfying \eqref{eq:conditiondata} and such that, for any $p\in [1,p_0]$, $\phi_1$ is not a $2p$-energy minimiser for $f$ in the class $\tp{id}+W^{1,2p}_0(B,B)$.
\end{enumerate}
\end{bigthm}

The local $L\log L$ integrability of non-negative Jacobians of $W^{1,n}$ maps was proved by \textsc{M\"uller} in~\cite{Muller1990}, and a global version, under suitable boundary regularity, was proved in \cite{Hogan2000}, while here \eqref{eq:estimateoverdomains} is the reverse inequality.

Condition \eqref{eq:averagecondition} is closely related to the regularity and symmetry properties of minimisers. Concerning the regularity, note that estimate \eqref{eq:estimateoverdomains} is in general false if $f$ does not satisfy \eqref{eq:averagecondition}: in fact, we showed in \cite{GuerraKochLindberg2020} that, for a Baire-generic $f$ in $X_p(B)$, solutions of \eqref{eq:jacdirichlet} are at best in $W^{1,p}(B)$, and not in $W^{1,2p}(B)$, as in \eqref{eq:estimateoverdomains}. Concerning the symmetry, part \ref{it:counterex} of Theorem \ref{bigthm:helein} shows that, in the absence of \eqref{eq:averagecondition}, energy minimisers are in general not symmetric. We conjecture that energy minimisers for the Dirichlet problem are also non-unique in general, compare with Theorem \ref{bigthm:manyminimisers} below. 

Let us now explain the role played by condition \eqref{eq:averagecondition} in studying energy minimisation. 
Writing $z=r e^{i \theta}$, note that
$$
|\D u|^2 = |\p_r u|^2 + \frac{|\p_\theta u|^2}{r^2} ,
\qquad 
\J u=  \frac{\p_\theta u}{r} \wedge \p_r u .
$$
Hence we see that energy minimisation favours maps for which:
\begin{enumerate}
\item\label{it:alignment}  $\p_r u$ is approximately perpendicular to $\p_\theta u$, so that $|\J u|\approx |\p_r u||\p_\theta u|/r$;
\item\label{it:equidistribenergy} $|\p_r u|\approx \frac 1 r |\p_\theta u|$, so that $|\p_r u |\frac{|\p_\theta u|}{r}\approx \frac 1 2(|\p_r u |^2+|\p_\theta u |^2/r^2)=\frac 1 2|\D u|^2$.
\end{enumerate}
Radial stretchings accomplish \ref{it:alignment} perfectly: indeed,
$$
\D\phi_1(z) = \frac{\rho(r)}{r} \tp{Id} + \bigg( \dot \rho(r)-\frac{\rho(r)}{r}\bigg)\frac{z\otimes z}{r^2}\quad \implies \quad 
\begin{cases}
\p_r \phi_1  = \dot\rho(r) z, \\ 
\frac 1 r \p_\theta \phi_1  = \frac{\rho(r)}{r} z^\bot.
\end{cases}
$$
There is, however, no reason for radial stretchings to satisfy \ref{it:equidistribenergy}, and this is where condition \eqref{eq:averagecondition} comes in: it is easy to see that this condition is equivalent to
\begin{equation}
\label{eq:radialaveragecond}
|\p_r \phi_1| \leq \frac 1 r |\p_\theta \phi_1|.
\end{equation}
The isoperimetric inequality shows that $\phi_1$ has optimal angular derivatives among all solutions of \eqref{eq:jac}, while \eqref{eq:radialaveragecond} ensures that these derivatives control the Dirichlet energy $|\D \phi_1|^2$, which enables us to prove that $\phi_1$ is an energy minimiser.

A simple sufficient criterion for \eqref{eq:radialaveragecond} to hold is that $r\mapsto f(r)$ is non-increasing. We also note that condition \eqref{eq:radialaveragecond} is not new: a radial stretching satisfying \eqref{eq:radialaveragecond} was called \textit{conformally non-expanding} in \cite{Iwaniec2012}, as it does not increase the conformal modulus of annuli.

\subsection{Energy minimisers in \texorpdfstring{$\R^n$}{}}

In this subsection we study \eqref{eq:jac} over $\R^n$. In this setting, the following question, essentially set by \textsc{Coifman}, \textsc{Lions}, \textsc{Meyer} and \textsc{Semmes} in \cite{Coifman1993}, remains an outstanding open problem:

\begin{question}\label{question:CLMS}
Is the Jacobian $\J \colon {\dot W}^{1,np}(\R^n,\R^n)\to \mathscr H^p(\R^n)$ surjective?
\end{question}

Here $\mathscr H^p(\R^n)$ stands for the real Hardy space and we refer the reader to \cite{Coifman1977, Stein2016} for its theory. We recall that, for $p>1$, $\mathscr H^p(\R^n)$ agrees with the usual Lebesgue space $L^p(\R^n)$. 

Question \ref{question:CLMS} is especially natural as, in \cite{Coifman1993,Hytonen2018}, the authors prove that $\mathscr H^p(\R^n)$ is the smallest Banach space containing the range of the Jacobian, compare with \cite{Lindberg2017} for the \textit{inhomogeneous} case. In \cite{Iwaniec1997}, see also \cite{Bonami2007}, \textsc{Iwaniec} went further than Question \ref{question:CLMS} by conjecturing:

\begin{conj}\label{conj:iwaniec}
For each $p\in [1,\infty)$, the Jacobian has a \textit{continuous right inverse}: there is a continuous map
$E\colon \mathscr{H}^p(\R^n)\to \dot{W}^{1,np}(\R^n,\R^n)$ such that $\J \circ E = \tp{Id}$.
\end{conj}

In \cite{Iwaniec1997}, \textsc{Iwaniec} proposed the following route towards Conjecture \ref{conj:iwaniec}:
\begin{para}{Strategy}\label{strat:iwaniec}
A possible way of proving Conjecture \ref{conj:iwaniec} is to establish the following claims:
\begin{enumerate}

\item\label{it:estimateminimizer} Every $np$-energy minimiser satisfies $\Vert \D u \Vert^n_{L^{np}(\R^n)} \lesssim \Vert \tp{J} u \Vert_{\mathscr{H}^p(\R^n)}$.

\item\label{it:uniquenessminimizer} For all $f \in \mathscr{H}^p(\R^n)$ there is a unique $np$-energy minimiser $u_f$ for $f$, modulo rotations.

\item\label{it:continuousdependence} For all $f\in \mathscr{H}^p(\R^n)$ there is a rotation $Q_f\in \tp{SO}(n)$ such that $f\mapsto Q_f u_f$ is continuous.
\end{enumerate}
\end{para}

The nonlinear open mapping principles that we proved in \cite{GuerraKochLindberg2020b} show that \ref{it:estimateminimizer} is \textit{equivalent} to a positive answer to Question \ref{question:CLMS}. In this direction, \textsc{Iwaniec} suggested that one should prove \ref{it:estimateminimizer} by constructing a Lagrange multiplier for every $np$-energy minimiser, see the third author's works \cite{Lindberg2020,Lindberg2015} for  results in this direction. 

Using the terminology of \cite{GuerraKochLindberg2020b}, we say that a solution in ${\dot W}^{1,n}(\R^n,\R^n)$ of \eqref{eq:jac} is \textit{admissible} if it is continuous and it satisfies the change of variables formula. We note that solutions in ${\dot W}^{1,np}(\R^n,\R^n)$, for $p>1$, are always admissible, see Remark \ref{rem:changeofvars}. The proof of Theorem \ref{bigthm:helein} is easily adapted to $\R^2$ and provides conditions on $f$ under which \ref{it:estimateminimizer} and \ref{it:uniquenessminimizer} hold:

\begin{bigcor}\label{bigcor:uniqueminimisers}
Let $f\in L^p(\R^2)$ be a radially symmetric function such that
$$ |f(r)|\leq \fint_{B_r(0)} f \d x\qquad \text {for a.e. } r\in (0,\infty).$$ 
Then $\Vert \D \phi_1 \Vert^2_{L^{2p}(\R^2)} \lesssim \Vert f \Vert_{\mathscr{H}^p(\R^2)}$ and, for $p>1$, $\phi_1$ is the {\normalfont unique} $2p$-energy minimiser for $f$, modulo rotations. For $p=1$ the same statement holds in the class of admissible solutions.
\end{bigcor} 

Further uniqueness results can be found in \cite{Lindberg2020}.
Nonetheless, and despite these positive results, this paper's second main contribution is to show that, in general, claim \ref{it:uniquenessminimizer} is false:
\begin{bigthm}\label{bigthm:manyminimisers}
Fix $1\leq p<\infty$. There is a radially symmetric function $f\in \mathscr H^p(\R^2)$ which has uncountably many $2p$-energy minimisers, modulo rotations.
\end{bigthm}

Theorem \ref{bigthm:manyminimisers} shows that energy minimisation is not a suitable selection criterion. It is also very difficult to work with energy minimisers directly: when $p=1$, we cannot decide whether they are admissible, although in \cite{GuerraKochLindberg2020b} we showed that, under natural assumptions, the existence of energy minimisers implies the existence of admissible solutions. 

To conclude the discussion of Strategy \ref{strat:iwaniec}, we note that, assuming \ref{it:estimateminimizer} holds, it remains to establish a nonlinear analogue of the classical Bartle--Graves theorem \cite{Bartle1952}. This theorem  states that a bounded linear surjection between Banach spaces has a bounded and continuous (but possibly nonlinear) right inverse, see \cite[page 86]{Bessaga1975} for a good overview. Without extra assumptions, the Bartle--Graves theorem does not generalise to multilinear mappings \cite{Fernandez1998}. 
However, one may use the results in \cite{GuerraKochLindberg2020b} and \cite{Jayne1985} to prove a partial result towards a nonlinear Bartle--Graves theorem for the Jacobian: assuming surjectivity of $\tp{J} \colon \dot{W}^{1,np}(\R^n,\R^n) \to \mathscr{H}^p(\R^n)$, there is a bounded right inverse that is continuous outside a meagre set, although we do not prove such a result here.

\subsection*{Outline}

This paper is structured as follows. In Section \ref{sec:radialdata} we consider the regularity of polar representations of a Sobolev map and we recall some useful formulae in polar coordinates. In Section \ref{sec:aclassofdata} we prove a more general version of Theorem \ref{bigthm:helein}\ref{it:averagecond} and in Section \ref{sec:counterexample} we prove Theorem \ref{bigthm:helein}\ref{it:counterex}. Section \ref{sec:nonuniqueness} contains the proof of Theorem \ref{bigthm:manyminimisers}.

\subsection*{Notation}

We use polar coordinates $z=r e^{i \theta}= x + iy \in \C\cong \R^2$ in the plane. We write $B_r(x)$ for the usual Euclidean balls in $\R^n$, and $\mb S_r\equiv \p B_r$ (when $x$ is omitted, it is understood that $x=0$). It is also useful to have notation for annuli: for $0<r<R$,
$$\mb A(r,R)\equiv \{z\in \C: r<|z|<R\}.$$
We will also abuse this notation slightly by setting $\mb A(0,R)\equiv B_R(0)$. Here $|z|$ denotes the Euclidean norm of $z\in \C$ and likewise for $A\in \R^{n\times n}$ we write $|A|\equiv \tp{tr}(A A^\tp{T})^\frac 1 2$ for the Euclidean norm. 
Finally, and unless stated otherwise, $p$ is a real number in $[1,+\infty)$.

\section{Polar coordinates and generalised radial stretchings}\label{sec:radialdata}

Given a planar Sobolev map $u\in W^{1,p}(\R^2,\R^2)$, we consider
polar coordinates both in the domain and in the target; that is, we
want to write
\begin{equation}
u(r e^{i \theta}) = \psi(r,\theta) \exp( i
\gamma(r,\theta))\label{eq:rep}
\end{equation}
for some functions $\psi\colon (0,\infty)\times[0,2\pi]\to [0,\infty)$
and $\gamma\colon (0,\infty)\times [0,2\pi] \to \R$, where furthermore we must have
the compatibility conditions
\begin{equation}
\label{eq:compatibilitypolarcoords}
\psi(r,0)=\psi(r,2\pi) \quad \tp{ and } \quad \gamma(r,0)-\gamma(r,2\pi)\in 2\pi \Z\qquad \tp{ for all }r.
\end{equation}
We will freely identify $(r,\theta)\equiv r e^{i \theta}$, adopting either notation whenever it is more convenient.

The existence of a representation as in \eqref{eq:rep} is a standard problem in lifting theory:

\begin{prop}\label{prop:lifting}
Let $0\leq R_1<R_2$ and $p\geq 2$. Let $u\in W^{1,p}(\mb A(R_1,R_2),\R^2)$ and, if $p=2$, suppose moreover that $u$ is continuous. Assume $u^{-1}(0)\subseteq \{0\}$. 
Then there are continuous functions 
$$\psi\in W^{1,p}\left([R_1,R_2]\times [0,2\pi]\right), \qquad \gamma\in W^{1,p}\left((\max\{R_1,\e\},R_2)\times [0,2\pi]\right),$$ where $\e\in(0,R_2)$ is arbitrary, which satisfy \eqref{eq:compatibilitypolarcoords} and such that the representation \eqref{eq:rep} holds. 

\end{prop}

\begin{proof}
Let $\e>0$ 
and consider the keyhole domains
\begin{align*}
&\mb A_{1,\e} \equiv [\max\{R_1,\e\},R_2]\times [\e,2\pi-\e],\\
&\mb A_{2,\e} \equiv [\max(R_1,\e),R_2]\times ([0,\pi-\e]\cup [\pi+\e,2\pi]).
\end{align*}
We freely identify $\mb A_{i,\e}$ with the respective domains in $\R^2$.

We first show the existence of a representation \eqref{eq:rep} in each $\mb A_{i,\e}$.
Note that if $u\in W^{1,p}$ then $\psi=|u|$ is also in $W^{1,p}$ and is continuous whenever $u$ is. Thus, since $0\not\in u(\mb A_{i,\e})$, it suffices to prove the existence of a continuous function $\gamma_i\in W^{1,p}(\mb A_{i,\e},\R)$ such that $u/|u|= e^{i \gamma_i}$ for $i=1,2$. Since $u$ is continuous, $u/|u|\in W^{1,p}(\mb A_{i,\e},\mb S^1)$, and so the existence of $\gamma_i$ follows from the results in \cite{Bethuel1988}, see also \cite{Bourgain2000}.

Thus, for almost every $(r,\theta)\in \mb A_{1,\e}\cap\mb A_{2,\e}$,
\begin{align*}
\psi(r,\theta)e^{i\gamma_1(r,\theta)}=u(r e^{i \theta})=\psi(r,\theta) e^{i\gamma_2(r,\theta)}\iff \gamma_1(r,\theta)-\gamma_2(r,\theta) = 2\pi k(r,\theta),
\end{align*}
where $k(r,\theta)\in \Z$. As $\gamma_1$, $\gamma_2$ are continuous in $\mb A_{1,\e}\cap \mb A_{2,\e}$, we must have \begin{align*}k(r,\theta)= \begin{cases}
k_1 \qquad \text{ for } \e<\theta<\pi-\e,\\
k_2 \qquad \text{ for } \pi+\e<\theta<\e.
\end{cases}
\end{align*}
Without loss of generality, upon redefining $\gamma_1$ we may assume $k_1=0$. Hence we may define
\begin{align*}
\gamma_\e(r,\theta)=\begin{cases}
					\gamma_1(r,\theta) &\text{ if } (r,\theta)\in\mb A_{1,\e},\\
					\gamma_2(r,\theta) &\text{ if } (r,\theta)\in\mb A_{2,\e}.
				\end{cases}	
\end{align*}
By a similar argument, we see that we may take $\gamma_\e=\gamma_\delta$ in $\mb A(R_1,R_2)\setminus (B_\delta\cup B_\e)$, so that in fact $u=\psi(r,\theta)e^{i\gamma(r,\theta)}$ with $\gamma\in W^{1,p}(\max(R_1,\e),R_2)\times [0,2\pi])$ for all $\e>0$. 
The conclusion follows.
\end{proof}

We remark that the conclusion of Proposition \ref{prop:lifting} is false if $p<2$, see \cite[§4]{Bourgain2000}. 

\begin{cor}
In the setting of Proposition \ref{prop:lifting}, we have a.e.\ the formulae
  \begin{gather}
\label{eq:jacobianpolarcoords}    \J u =\frac{1}{2r} \frac{\p(\psi^2,\gamma)}{\p(r,\theta)} = \frac{1}{2r}\left(\p_r(\psi^2) \p_\theta \gamma -\p_\theta(\psi^2)\p_r \gamma\right),\\
\label{eq:dirichletenergypolarcoords}    |\D u|^2=|\p_r \psi|^2+|\psi \p_r \gamma|^2 
    +\frac{|\p_\theta \psi|^2}{r^2} +\frac{|\psi
      \p_\theta \gamma|^2 }{r^2}.
  \end{gather}
\end{cor}

\begin{proof}
It is not difficult to formally derive the above formulae whenever the representation \eqref{eq:rep} holds. To make the argument rigorous it suffices to note that, due to the regularity of $\psi$ and $\gamma$, the right-hand sides in \eqref{eq:jacobianpolarcoords}--\eqref{eq:dirichletenergypolarcoords} define locally integrable functions. Thus the corollary follows by a standard density argument.
\end{proof}

A function $f\colon B_R(0)\to \R$ is said to be \textit{radially symmetric} if $|x|=|y|\implies f(x)=f(y)$ and we identify any such function with a function $f\colon [0,+\infty)\to \R$ in the obvious way. For such a function, it is natural to look for solutions of \eqref{eq:jac} possessing some symmetry, in particular satisfying $\p_\theta \psi=0$ and $\p_r \gamma = 0$ if a representation as in \eqref{eq:rep} holds:
\begin{ndef}\label{def:radialstretching} 
The class of \textit{generalised radial stretchings} consists of maps of the form
\[
\phi_k(z)\equiv \frac{\rho(r)}{\sqrt{|k|}} e^{i k \theta}
 \]
 where $k\in \Z\setminus\{0\}$ is the topological degree of the map and $\rho\geq 0$. If $k=1$ we refer to such maps simply as radial stretchings.
\end{ndef}
Generalised radial stretchings are spherically symmetric in the sense that they map circles centred at zero to circles centred at zero. The following is a useful criterion concerning the Sobolev regularity of generalised radial stretchings:

\begin{lemma}\label{lemma:ball}
Let $p\in [1,\infty)$ and $k\in \Z\setminus\{0\}$. Given $R\in (0,+\infty]$, $\phi_k\in {\dot W}^{1,p}(B_R(0), B_R(0))$  if and only if $\rho$ is absolutely continuous on $(0,R)$ and 
$$\Vert \D \phi_k \Vert_{L^p(B_R(0))}^p\approx 
\int_0^R \left(\bigg|\frac{\dot \rho(r)}{k}\bigg|^p + \bigg|k\, \frac{\rho(r)}{r}\bigg|^p \right)r \d r<\infty.$$
\end{lemma}

We omit the proof of the lemma, as it is a straightforward adaptation of \cite[Lemma 4.1]{Ball1982}. 

It is not the case that any radially symmetric $f\in \mathscr H^p(\R^2)$ admits generalised radial stretchings as solutions of \eqref{eq:jac}. Indeed, from \eqref{eq:jacobianpolarcoords}, formally we see that 
\begin{equation}
\label{eq:defrho}
\J \phi_k = f \quad \implies \quad \rho(r) = \sqrt{\frac 1 k\int_0^r 2 s f(s)\d s}
\end{equation}
and hence, for the equation $\J \phi_k=f$ to be solvable for some $k\in \Z \exc\{0\}$, we must have
\begin{equation}
\tp{either } \int_0^r 2 s f(s)\d s\leq 0\tp{ for a.e.\ } r, \qquad \tp{or }
\int_0^r 2 s f(s)\d s\geq 0\tp{ for a.e.\ } r.
\label{eq:orientation}
\end{equation}
Conversely, whenever $f$ satisfies \eqref{eq:orientation}, we will take $\rho$ as in \eqref{eq:defrho}, so that $\phi_k$ is a \textit{formal} solution of $\J\phi_k =f$. Indeed, note that \eqref{eq:orientation} is not enough to ensure the existence of generalised radial stretching solutions with the required regularity: 

\begin{ex}\label{ex:annulus}
If $f=1_{\mb A(1,2)}$ then, for any $k\in \N\exc\{0\}$, $\phi_k$ is in $\bigcup_{1\leq q <2} W^{1,q}\exc W^{1,2}(B_2,\R^2)$.
\end{ex}

The claim in Example \ref{ex:annulus} follows  readily from Lemma \ref{lemma:ball}.
In the next section we find a condition on $f$ which ensures that $\phi_k$ has ${\dot W}^{1,2p}$-regularity. For other related results see  \cite{GuerraKochLindberg2020}, \cite[§3]{Lindberg2015} and \cite[§7]{Ye1994}.

\section{A class of data with symmetric energy minimisers}
\label{sec:aclassofdata}

The key step in establishing Theorem \ref{bigthm:helein} is the following more general proposition which may be of independent interest.

\begin{prop}\label{prop:energyquasiminim}
Let $p\in [1,\infty)$ and $f\in \mathscr H^p(\R^2)$ be a radially symmetric function such that, for some $\lambda\geq 1$ and a.e.\ $r\in (0,+\infty)$,
\begin{equation}\label{eq:lambdadata}
| f(r)|\leq  \lambda \fint_{B_r(0)} f  \d x.
\end{equation}
 Let $\phi_1$ denote the radial stretching solving $\J\phi_1=f$.

\begin{enumerate}
\item\label{it:improvedintegrability} We have the estimate
\begin{equation}
\Vert \D \phi_1 \Vert_{L^{2p}(\R^2)}^2 \leq 
C(\lambda)\, \Vert f \Vert_{\mathscr H^p(\R^2)}.
\label{eq:improvedintegrability}
\end{equation}
\item\label{it:quasiminimiser} Let $u\in W^{1,2}_\tp{loc}(\R^2, \R^2)$ be a solution of
$\J u =f$  
such that, for a.e.\ $r\in (0,+\infty)$,
\begin{equation}
\label{eq:isoperimetric}
4 \pi \int_{B_r} \J u \d x \leq \biggr(\int_{\mb S_r}|\D u\cdot \nu^\bot|\d\theta\biggr)^2.
\end{equation}
Then, with $Z$ denoting the {\normalfont Zhukovsky function} $Z(\lambda)\equiv \frac 1 2 \left(\frac 1 \lambda + \lambda\right)$, we have the estimate
\begin{equation}
\label{eq:minimising}
\int_{\mb S_r} |\D \phi_1|^{2p} \leq Z(\lambda)\int_{\mb S_r} |\D u|^{2p}
\end{equation} 
for a.e.\ $r\in (0,\infty).$
\item\label{it:uniqueness} If $\lambda= 1$ and if \eqref{eq:minimising} holds with equality then \eqref{eq:isoperimetric} holds with equality.
\end{enumerate} 
\end{prop}

In the statement of the theorem, as well as in its proof, $u$ denotes the precise representative of the equivalence class $[u]\in W^{1,2p}_\tp{loc}$. We refer the reader to \cite{Evans2015} for the definition and properties of precise representatives. 

We note that condition \eqref{eq:isoperimetric} is a parametric version of the isoperimetric inequality. In particular, it holds under natural assumptions including the setting of Theorem \ref{bigthm:helein}, see already Proposition \ref{prop:isoperimetric}. 

Before proceeding with the proof, it is useful to note that condition (\ref{eq:lambdadata}) can be rewritten in terms of $\phi_1$ as
\begin{equation}
\label{eq:lambdaradial}
 |\dot \rho(r)|\leq  \lambda \left| \frac{\rho(r)}{r} \right|,
\end{equation}
see also the discussion in the introduction.
 It is worth mentioning that, in \eqref{eq:lambdadata}, we make implicitly a \textit{choice of orientation}. Indeed, in order to ensure the existence of generalised radial stretchings solving the equation, it must be the case that the map $r\mapsto \int_{B_r} f \d x$ does not change sign, see \eqref{eq:orientation}. Clearly  \eqref{eq:lambdadata} implies that this map is non-negative. There is an analogue of Proposition \ref{prop:energyquasiminim} in the case where $\int_{B_r} f \d x$ is always non-positive: in that case, we replace $\phi_1$ with $\phi_{-1}$.

\begin{proof}[Proof of Proposition \ref{prop:energyquasiminim}\ref{it:improvedintegrability}]
For the case $p>1$, we combine \eqref{eq:lambdaradial} with Lemma \ref{lemma:ball} to get
\begin{equation*}
\begin{split}
\Vert \D \phi_1 \Vert_{L^{2p}(\R^2)}^{2p}
&\approx \int_0^\infty \left(|\dot \rho(r)|^{2p} + \left|\frac{\rho(r)}{r}\right|^{2p} \right)r \d r
 \lesssim_\lambda \int_0^\infty \left|\frac{\rho(r)}{r}\right|^{2p} r \d r
= \int_{\R^2}\left|\fint_{\bar{B}_{|x|}(0)} f \d y\right|^p \d x.
\end{split}
\end{equation*}
Denoting by $M$ be the (non-centred) Hardy--Littlewood maximal function, we have
$$\Vert \D \phi_1 \Vert_{L^{2p}(\R^2)}^{2p}
 \lesssim_\lambda \int_{\R^2} \left|M f(x)\right|^p \d x\lesssim \int_{\R^2} |f(x)|^p \d x,$$
as wished.

For the case $p=1$, we need to argue in a more careful way and we use the fact that
$$\Vert f(|r|) \Vert_{\mathscr H^1(\R,|r|\d r)} \approx \Vert f(|x|) \Vert_{\mathscr H^1(\R^2,\d x)},$$
see the proof of \cite[Corollary (2.27)]{Coifman1977}. Recall that an $\mathscr H^1(\R, |r|\d r)$-atom is simply a function $a\colon \R \to \R$ such that
$$\supp a\subset [r_1,r_2], \qquad \Vert a \Vert_\infty\leq \frac{1}{\int_{r_1}^{r_2} |s| \d s}, \qquad \int_\R a(r) |r| \d r=0,$$ 
for some real numbers $r_1<r_2$, and that moreover for any $f\in \mathscr H^1(\R,|r|\d r)$ there exist atoms $a_i$ and real numbers $\lambda_i\in \R$ such that
\begin{equation} \label{eq:atomicdecomposition}
0=\lim_{N \to \infty} \left\Vert f - \sum_{i=1}^N \lambda_i a_i \right\Vert_{\mathscr H^1(\R, |r| \d r)}, \qquad \sum_{i=1}^\infty |\lambda_i| \lesssim \Vert f \Vert_{\mathscr H^1(\R,|r| \d r)}.
\end{equation}
Arguing as in the case $p>1$ we see that
$$\Vert \D \phi_1 \Vert_{L^{2p}(\R^2)}^{2p} \lesssim_\lambda \int_0^\infty \frac{1}{r} \int_0^r 2 \, f(s) s \d s\d r = \lim_{\e \to 0} \int_\e^{1/\e} \frac{1}{r} \int_{-r}^r f(s)|s|\d s\d r,$$
where we also used $f=f(|r|)$ in the last equality.

Fix $\e > 0$. By using \eqref{eq:atomicdecomposition} and the dominated convergence theorem,
\[\int_\e^{1/\e} \frac{1}{r} \int_{-r}^r f(s)|s|\d s\d r = \lim_{N \to \infty} \sum_{j=1}^N \lambda_j \int_\e^{1/\e} \frac{1}{r} \int_{-r}^r a_j(s)|s|\d s\d r.\]
When $N \in \N$, suppose $a$ is one of the atoms $a_1,\ldots,a_N$ and let $0\leq \tilde r_1< \tilde r_2$ be, respectively, the minimum and the maximum of $|\cdot |$ over $[r_1,r_2]$.
Then
\begin{align*}
\int_\e^{1/\e} \frac{1}{r} \int_{-r}^r a(s)|s|\d s\d r
&= \int_{\max\{\e,\tilde r_1\}}^{\min\{1/\e,\tilde r_2\}} \frac{1}{r} \int_{-r}^r a(s)|s|\d s\d r
\\ &\leq \frac{\int_{\tilde r_1}^{\tilde r_2} \frac{1}{r} \int_{-r}^r|s|\d s\d r}{\int_{r_1}^{r_2} |s| \d s}
=
 \frac{\tilde r_2^2-\tilde r_1^2}{2\int_{r_1}^{r_2} |s| \d s}\leq 1.
\end{align*}
By letting first $N \to \infty$ and then $\e \to 0$, the conclusion follows from \eqref{eq:atomicdecomposition}.
\end{proof}

Before finishing the proof of Proposition \ref{prop:energyquasiminim}, we record the following elementary lemma:
\begin{lemma}\label{lemma:psi}
Define $\psi\colon(0,\infty)\times \R\to\R$ by $\psi(a,b)\equiv a+b^2/a$. Then
  \begin{enumerate}
  \item\label{it:cvx} the function $\psi$ is convex;
  \item\label{it:min} for each $b\in \R$, the function $a\mapsto \psi(a,b)$ is increasing in $(0,b)$ and decreasing in $(b,+\infty)$ and it has a global minimum at $a=|b|$;
  \item \label{it:zhukovsky} for $\lambda>0$, if $a_2\leq a_1$ and $|b|\leq \lambda a_2$ then
  $\psi(a_2,b)\leq Z(\lambda) \psi(a_1,b)$.
  \end{enumerate}
 \end{lemma}
 \begin{proof} 
The first two properties are readily checked.
To prove \ref{it:zhukovsky}, note that when $|b|\leq a_2$ the conclusion follows from \ref{it:min}, since $1\leq Z(\lambda)$. When $a_2< |b|$ then, by applying \ref{it:min} twice,
\begin{equation*}
\psi(a_2,b)\leq \psi(b/\lambda,b)=Z(\lambda) \psi(b,b) \leq Z(\lambda)\psi(a_1,b).\qedhere
\end{equation*}
\end{proof}

\begin{proof}[Completion of the proof of Proposition \ref{prop:energyquasiminim}]
We first deal with the case $p=1$. Note that $\phi_1$ is continuous and denote also by $u$ the precise representative of the class $[u]\in W^{1,2}_\tp{loc}$. Consider the set of ``good'' radii
$$\mc G \equiv \biggr\{r\in (0,\infty): 
\begin{array}{l}
 u|_{\mb S_r} \tp{ is absolutely continuous, (\ref{eq:lambdadata}) and \eqref{eq:isoperimetric} hold,} \\
\tp{and }\J u(x) = f(x) \tp{ for } \mathscr H^1\tp{-a.e. } x\in \mb S_r
\end{array} \biggr\}.$$
Since $u$ is a Sobolev function, our hypotheses together with an application of Fubini's theorem show that the $\mc G$ has full measure, i.e.\ $\mathscr L^1(\R^+\exc \mc G)=0$. 

Fix $r\in \mc G$. The crucial observation is that $\phi_1$ satisfies the isoperimetric inequality \eqref{eq:isoperimetric} with equality: this is easily checked, but it can also be seen as a consequence of the fact that $\phi_1$ maps circles to circles and has degree one. Hence, as $u$ satisfies \eqref{eq:isoperimetric} by assumption, 
$$
\left(\int_{\S_r} |\D \phi_1\cdot
    \nu^\bot|\d\theta\right)^2 = 4 \pi\int_{B_r} f \d x 
    \leq \biggr(\int_{\S_r} |\D u \cdot \nu^\bot|\d\theta\biggr)^2.
$$
Moreover, $\D\phi_1(re^{i \theta})$ is constant on $\mb S_r$ and so, using Jensen's inequality, we arrive at
\begin{equation}
\label{eq:4}
   |\D \phi_1 \cdot \nu^\bot|^2 = \left(\fint_{\S_r} |\D \phi_1\cdot
    \nu^\bot|\d\theta\right)^2
    \leq \fint_{\S_r} |\D u \cdot \nu^\bot|^2\d\theta.
\end{equation}
Here we implicitly assume that $\D\phi_1$ and $\D u$ are evaluated at the point $x=r e^{i\theta}$, in order to lighten the notation; the same convention is used in the rest of the proof.

We now note the following cofactor identity: if $\nu\in \mb S^{1}$ and $A\in \R^{2\times 2}$, 
  $$\det A = \det A \langle \nu, \nu \rangle = \langle
  \tp{cof}(A)^\tp{T} A \nu, \nu \rangle = \langle A \nu, \tp{cof}(A)
  \nu\rangle.$$
  Using the Cauchy--Schwarz inequality and the fact that $|\tp{cof}(A)\nu|=|A\nu^\bot|$,
we have
  \begin{equation}
  \det A \leq |A\nu| |\tp{cof}(A) \nu | \hs \implies\hs
  |A\nu^\bot|^2 +\frac{(\det A)^2}{|A\nu^\bot|^2}\leq |A\nu^\bot|^2+
  |A\nu|^2= |A|^2.\label{eq:CS}
\end{equation}
We apply (\ref{eq:CS}) to $A=\D u(x)$, choosing $\nu=x/r$: since $\J u=f$,
  \begin{align}
  \label{eq:phi1}
  \fint_{\S_r}\psi(|\D u\cdot \nu^\bot|^2, f(r))\d\theta
 	= \fint_{\S_r} |\D u \cdot\nu^\bot|^2 + 
 	\frac{ f^2}{ |\D u \cdot\nu^\bot|^2} \d\theta\leq   \fint_{\S_r} |\D u|^2\d\theta,
  \end{align}
  where $\psi$ is as in Lemma \ref{lemma:psi}.
    By Lemma \ref{lemma:psi}\ref{it:cvx},  Jensen's inequality applies to yield 
  \begin{equation}
\label{eq:phi2}
\psi\left(\fint_{\S_r} |\D u\cdot\nu^\bot|^2\d\theta, f(r)\right)=
 \psi\left(\fint_{\S_r} |\D u\cdot\nu^\bot|^2\d\theta, \fint_{\S_r} f(r)\d\theta\right)\leq
 \fint_{\S_r} \psi(|\D u\cdot \nu^\bot|^2,f(r)) \d\theta.
 \end{equation}
Let us also note that we have equality in (\ref{eq:phi1}) whenever we have equality in (\ref{eq:CS}), i.e.\ whenever we have equality in Cauchy--Schwarz. In other words, we have equality in (\ref{eq:phi1}) if and only if $\tp{cof}(\D u)\nu $ is parallel to $\D u\cdot\nu$ (or, equivalently, if and only if $\p_r u \bot \p_\theta u$), which is the case if $u$ is a radial stretching\footnote{Although this is not important for our purposes, it is also the case if $u$ is conformal.}, see also the discussion in the Introduction.

We now take 
 $$a_1 =\fint_{\mb S_r}|\D u\cdot \nu^\bot|^2\d \theta, \qquad 
 a_2=|\D \phi_1\cdot \nu^\bot|^2=\frac{\rho^2(r)}{r^2} ,\qquad 
 b=f(r)=\frac{\rho(r) \dot \rho(r)}{r}.$$ From (\ref{eq:4}) we have that $a_2\leq a_1$ and from (\ref{eq:lambdaradial}) we have $|b|\leq \lambda a_2$.
 Hence Lemma \ref{lemma:psi}\ref{it:zhukovsky}, combined with (\ref{eq:phi1}) and (\ref{eq:phi2}), gives
 \begin{align*}
\psi\left(|\D \phi_1\cdot\nu^\bot|^2, f(r)\right)
    \leq Z(\lambda)\, \psi\left(\fint_{\S_r} |\D u\cdot\nu^\bot|^2\d\theta, f(r)\right)
   \leq Z(\lambda) \fint_{\S_r} |\D u|^2\d\theta.
     \end{align*}
As noted above, $\phi_1$ satisfies \eqref{eq:phi1} with equality and so 
$$\fint_{\S_r} |\D \phi_1|^2\d \theta = |\D \phi_1|^2= \psi\left(|\D \phi_1\cdot\nu^\bot|^2, f(r)\right).$$
This proves \eqref{eq:minimising} when $p=1$.

The case $p>1$ follows from the case $p=1$: since $x\mapsto x^{2p}$ is a strictly convex, increasing function over $\R^+$,  we can apply Jensen's inequality to conclude that
$$\fint_{\mb S_r} |\D \phi_1|^{2p}\d\theta=
\left(\fint_{\mb S_r}|\D \phi_1|^2 \d\theta\right)^{p} \leq 
\left(\fint_{\mb S_r}|\D u|^2\d\theta\right)^{p}\leq \fint_{\mb S_r} |\D u|^{2p}\d\theta,$$
where we also used the fact that $\D \phi_1$ is constant in $\mb S_r$ in the first equality.

Finally, \ref{it:uniqueness} follows by inspection of the proof. Since $\psi(a_2,b)<\psi(a_1,b)$ if $b\leq a_2<a_1$, to have equality in \eqref{eq:minimising} we must have $a_2=a_1$, that is, we must also have equality in \eqref{eq:isoperimetric}.
 \end{proof}

\begin{remark}
The dependence on $\lambda$ in the estimate \eqref{eq:improvedintegrability} is not uniform. That this must be the case is easily seen by considering regularised versions of Example \ref{ex:annulus}, see also Section \ref{sec:counterexample}. 
\end{remark} 

We next show that \eqref{eq:isoperimetric} holds under natural assumptions.
\begin{prop}\label{prop:isoperimetric}
Fix $p\in[1,\infty)$ and $R>0$. Let $u\in W^{1,2p}(B_R(0),\R^2)$ be a continuous map such that $\J u =f$ a.e.\ in $B_R(0)$. Suppose furthermore that for a.e.\ $r\in(0,R)$ the change of variables formula
\begin{equation}
\label{eq:changeofvars}
\int_{B_r} \J u \d x = \int_{\R^2} \tp{deg}(y,u,B_r) \d y,
\end{equation}
holds.
Then \eqref{eq:isoperimetric} holds for a.e.\ $r\in (0,R)$. Moreover,  equality holds in \eqref{eq:isoperimetric} if and only if $u(\mb S_r)$ is a circle which is traversed one time.
\end{prop}

In \eqref{eq:changeofvars}, $\tp{deg}(y,u,B_r)$ denotes the \textit{topological degree} of $u$ at $y$ with respect to $B_r$.

\begin{proof}
We first note that due to the Sobolev regularity of $u$, $u(\mb S_r)$ is a continuous rectifiable curve for almost every $r\in (0,R)$ and hence we may restrict to such $r$ without loss of generality.
We now recall the following \textit{generalised isoperimetric inequality}: given a continuous rectifiable curve $\Gamma$, let $(E_k)_k$ be the components of $\R^2\exc \Gamma$; on each $E_k$, $\Gamma$ has a well-defined winding number $w_k$. Then we have
\begin{equation}
\label{eq:generalisedisoperimetric}
4\pi \sum_k w_k^2\, \mathscr L^2(E_k) \leq l(\Gamma)^2,
\end{equation}
with equality if and only if $\Gamma$ is a circle traversed a finite number of times in a given direction. Here $l(\Gamma)$ denotes the length of $\Gamma$.
This inequality was proved implicitly in \cite[page 487]{Federer1960} and then later in \cite{Banchoff1971}, but see also \cite{Osserman1978} for a comprehensive overview.

We want to apply \eqref{eq:generalisedisoperimetric} when $\Gamma\colon \mb S^1\to \R^2$ is the curve $\Gamma(\theta)=u(re^{i \theta})$. Recall that, at a point $y$, the winding number of the curve $\Gamma$ with respect to $y$ is just $\tp{deg}(y,u,B_r)$, see for instance \cite[§6.6]{Deimling1985}. Since $l(\Gamma)= \int_{\mb S_r} |\tp{cof}(\D u)\nu|\d \theta$, we can use \eqref{eq:generalisedisoperimetric} to get
$$
\int_{\R^2} \tp{deg}(y,u,B_r)^2 \d y\leq 
\frac{1}{4\pi} \biggr(\int_{\mb S_r} |\tp{cof}(\D u)\,\nu|\d\theta\biggr)^2.
$$
As the topological degree is an integer, we deduce from \eqref{eq:changeofvars}
that
\begin{equation}
\int_{B_r} \J u\d x\leq \frac{1}{4\pi} \biggr(\int_{\mb S_r} |\tp{cof}(\D u)\,\nu|\d\theta\biggr)^2.
\label{eq:finalestimate}
\end{equation}
This proves \eqref{eq:isoperimetric}, since $|\tp{cof}(A)\nu|=|A\nu^\bot|$ for $A\in \R^{2\times 2}$.

The equality cases follow from the equality cases for \eqref{eq:generalisedisoperimetric} together with the fact that we must have $\tp{deg}(y,u,B_r)=\pm 1$ for $y\in u(B_r)$ to get equality in \eqref{eq:finalestimate}.
\end{proof} 

\begin{remark}\label{rem:changeofvars}
For $p>1$, the continuity assumption in Proposition \ref{prop:isoperimetric} is not restrictive and moreover \eqref{eq:changeofvars} also holds automatically, as maps in a supercritical Sobolev space always satisfy the Lusin (N) property. We refer the reader to \cite{Fonseca1995,Hencl2014a} for further details. 

For $p=1$ it is not in general the case that solutions are continuous and satisfy \eqref{eq:changeofvars}. However, both properties are satisfied over open sets where $f>0$ a.e., as in this case solutions have finite distortion. Assuming a positive answer to Question \ref{question:CLMS}, one can always find solutions satisfying both properties over bounded domains where $f\geq 0$ a.e.\ \cite[Theorem C]{GuerraKochLindberg2020b}.
\end{remark} 
  
We conclude this section by showing how Theorem \ref{bigthm:helein}\ref{it:averagecond} follows from Proposition \ref{prop:energyquasiminim}.
 
\begin{proof}[Proof of Theorem \ref{bigthm:helein}\ref{it:averagecond}]
Fix $p\in [1,\infty)$.
Since \eqref{eq:conditiondata} holds, Proposition \ref{prop:isoperimetric} applies. As $f$ satisfies \eqref{eq:lambdadata} with $\lambda=1$,  we conclude from \eqref{eq:minimising} that $\phi_1$ is a $2p$-energy minimiser.

That $\phi_1$ is the unique $2p$-energy minimiser in $\tp{id}+W^{1,2p}_0(B,B)$ follows from Proposition \ref{prop:energyquasiminim}\ref{it:uniqueness} and the equality case of Proposition \ref{prop:isoperimetric}. Indeed, for any $2p$-energy minimiser $u$ we must have that, for a.e.\ $r\in (0,R)$, $u(\mb S_r)$ is a circle; that is, using Proposition \ref{prop:lifting}, we may write
$$u(r e^{i \theta}) = \psi (r) e^{i \gamma(r,\theta)}.$$
Since $\tp{deg}(y,u,B_r)=1$ for all $y\in u(B_r)$, as $f$ is positive a.e., we see from \eqref{eq:jacobianpolarcoords} that 
$$\psi(r)=\sqrt{\int_0^r 2 s f(s)\d s }=\rho(r).$$
It is now easy to see from \eqref{eq:dirichletenergypolarcoords} that, as $u$ is a minimiser, $\p_r \gamma (r,\theta)= 0$ for a.e.\ $r\in (0,R)$. It follows from the boundary condition $u=\tp{id}$ on $\p B$ that $u=\phi_1$.

Finally we prove \eqref{eq:estimateoverdomains}. Arguing as in the proof of the case $p>1$ of Proposition \ref{prop:energyquasiminim}\ref{it:improvedintegrability}, we have that
$$\int_B |\D \phi_1|^{2p}\d x \lesssim\int_B |M f(x)|^p \d x,$$
for any $p\in [1,\infty)$. Through the maximal inequality this immediately implies \eqref{eq:estimateoverdomains} for $p>1$. To deal with the endpoint $p=1$ we recall that, whenever $\supp f\subset \bar{B}$, then
$$\Vert f \Vert_{L\log L(B)} \approx \Vert Mf\Vert_{L^1(B)},$$
see for instance \cite[page 23]{Stein1970}. Extending $f\in X_p(B)$ by zero outside $B$ we finish the proof.
\end{proof} 

Inspecting the above proof we also readily obtain Corollary \ref{bigcor:uniqueminimisers}.
We also note that the proof of Theorem \ref{bigthm:helein}\ref{it:averagecond} does not use any information about the behaviour of solutions on the boundary of the domain. It would be interesting to know the extent to which the boundary condition impacts the symmetry of energy minimisers. A model problem in this direction is to consider, for $\e>0$, the datum $f_\e(r)\equiv c_\e r^\e$, where $c_\e\equiv \frac{2}{2+\e}$ is such that $\fint_{B_1} f_\e \d x = 1$. It is easy to see that
$$f_\e(r)= \frac{2+\e}{2} \fint_{B_r(0)} f \d x$$
and hence \eqref{eq:minimising} shows that, for $\e\ll 1$, the energy of the corresponding radial stretchings is arbitrarily close to that of any other energy minimiser. However, we do not know whether the corresponding radial stretchings are $2p$-energy minimisers in $\tp{id}+W^{1,2p}_0(B,B)$.
 
\section{Non-symmetric energy minimisers}
\label{sec:counterexample}

In this section we prove part \ref{it:counterex} of Theorem \ref{bigthm:helein}. 
For a point $z=(x,y)\in \R^2$, let us write $|z|_1\equiv|x|+|y|$ for its $\ell^1$-norm and $$Q_r\equiv \{z\in \R^2:|z|_1<r\},\qquad \mb A_1(r,R)\equiv \{z\in \R^2:r<|z|_1<R\}$$
for the corresponding balls and annuli. The following example, although simple, is useful:

\begin{ex}[Mapping a ball onto a square]\label{ex:balltosquare}
The map
$$\eta(x,y)\equiv \frac{r\,\tp{sgn}(x)}{\sqrt 2}\,\begin{cases}
(1,4/\pi\arctan(y/x) & \tp{if } |y|< |x|,\\
(4/\pi \arctan(x/y),1)& \tp{if } |y|\geq |x|,\\
\end{cases}$$
is bi-Lipschitz and satisfies a.e.\ $\det\D \eta= 2/\pi$. For any $r>0$, we also have $R\circ\eta(B_r)=Q_r$, where $R$ is a rotation by angle $\frac \pi 4$.
\end{ex}

\begin{figure}[htbp]
\centering
\includegraphics[width=0.6\textwidth]{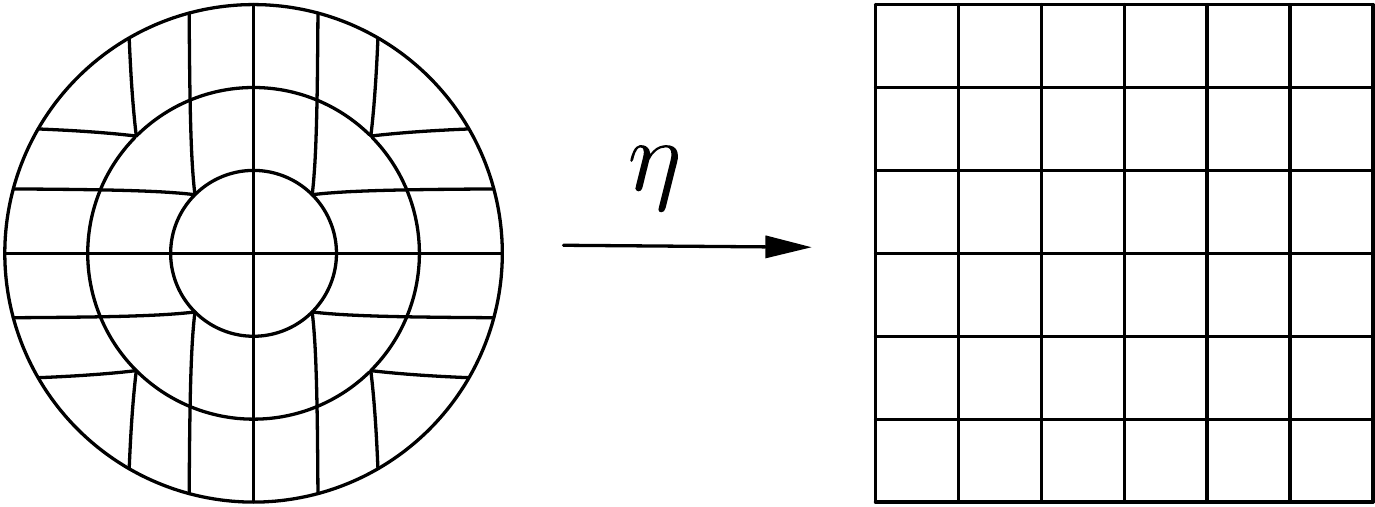}
\label{fig:balltosquare}
\caption{The map from Example \ref{ex:balltosquare}.}
\end{figure}

The map in Example \ref{ex:balltosquare} can be found in \cite{Griepentrog2008}. In fact, Example \ref{ex:balltosquare} is an explicit particular case of a more general construction, due to \textsc{Fonseca}--\textsc{Parry} \cite[Theorem 5.4]{Fonseca1992}. 
Their result applies to all domains of the following class:
\begin{ndef}
\label{def:classC}
A domain $\Omega\subset \R^n$ is of class $\mathscr C$ if there are $\e,\delta>0$ and $N\in \N$ such that:
\begin{enumerate}
\item $B_\e(0)\subset \Omega$ and $\Omega$ is bounded and star-shaped with respect to $0$, that is, every ray starting at 0 intersects $\p \Omega$ exactly once;
\item there is a finite partition $\Omega=\bigcup_{i=1}^N \Omega_i$ such that each $\Omega_i$ is a cone with vertex at 0, $B_\e(0)\cap \Omega_i$ is convex, $\p\Omega_i\cap \p \Omega$ is $C^1$ and satisfies $\nu(x)\cdot x\geq \delta$ for all $x\in \p \Omega_i\cap \p \Omega$, where $\nu$ denotes the outward unit normal.
\end{enumerate}
\end{ndef}

Given two domains $\Omega, \tilde \Omega$ of class $\mathscr C$, as they are star-shaped with respect to 0, there is a unique Lipschitz function $\psi\colon \p \Omega\to (0,+\infty)$ such that $\psi(x) x\in \p \tilde \Omega$ for all $x\in \p \Omega$. The next theorem was proved in \cite{Fonseca1992}, although the statement here is more precise than theirs:

\begin{thm}\label{thm:FonsecaParry}
Let $\Omega,\tilde \Omega$ be two domains of class $\mathscr C$. Then there is a surjective map $v\colon \Omega\to \tilde \Omega$ 
which is $L$-bi-Lipschitz, i.e.
$$\frac{1}{L} |x-y|\leq |v(x)-v(y)|\leq L |x-y| \quad \tp{ for all } x,y\in \overline \Omega,$$
and which solves, for $\psi$ as above,
$$
\begin{cases}
\J v = |\tilde \Omega|/|\Omega| & \text{in } \Omega,\\
v(x)=\psi(x) x & \text{for } x\in \p \Omega.
\end{cases}
$$
 Moreover, $L>0$ is a constant which depends only on $\delta,\e,n,N,\tp{diam}(\Omega)$ and $\tp{diam}(\tilde \Omega)$.
\end{thm}

Our goal is to use Theorem \ref{thm:FonsecaParry} to prove Theorem \ref{bigthm:helein}\ref{it:counterex}. If we do not require $f$ to be bounded away from zero, the following yields a simple example:
\begin{ex} Let
$f=\frac 4 3 1_{\mb A(1,2)}$ and
note that the radial stretching $\phi_1$ solving $\J \phi_1=f$ is not in $W^{1,2}(B_{1+\delta}(0))$, for any $\delta>0$, c.f.\ \eqref{eq:radialblowsup}. Actually, it is a general fact that $W^{1,2}$ solutions of \eqref{eq:jac} cannot be constant in open sets where $f=0$, for otherwise they would have integrable distortion and hence would be open mappings. 

We can apply Theorem \ref{thm:FonsecaParry} to the domains
$$\Omega= \mb A(1,2)\cap \{x>0,y>0\},\qquad \tilde \Omega= B_2(0) \cap \{x>0,y>0\},$$
which are star-shaped with respect to $(1,1)$,
to find a bi-Lipschitz map $u\colon \Omega\to \tilde \Omega$ which has constant Jacobian in $\Omega$. One can then extend $u$ to $B_1(0)\cap\{x>0,y>0\}$ in a trivial way, using the boundary data on the arc $\mb S_1\cap \{x>0,y>0\}$, and then extend $u$ to $B_2(0)$ through reflections along the axes, i.e.\ by setting
\begin{equation}
\label{eq:reflections}
u(x,y)=\begin{cases}
(u^1(x,-y), -u^2(x,-y)) & \tp{if } x>0, y<0,\\
(-u^1(-x,y), u^2(-x,y)) & \tp{if } x<0, y>0,\\
(-u^1(-x,-y), -u^2(-x,-y)) & \tp{if } x<0, y<0,
\end{cases}
\end{equation}
see Figure \ref{fig:epsilonzero}. Hence there is a Lipschitz solution $u\colon B_2(0)\to B_2(0)$ of \eqref{eq:jacdirichlet}.
\end{ex}

\begin{figure}[htbp]
\centering
\includegraphics[width=0.9\textwidth]{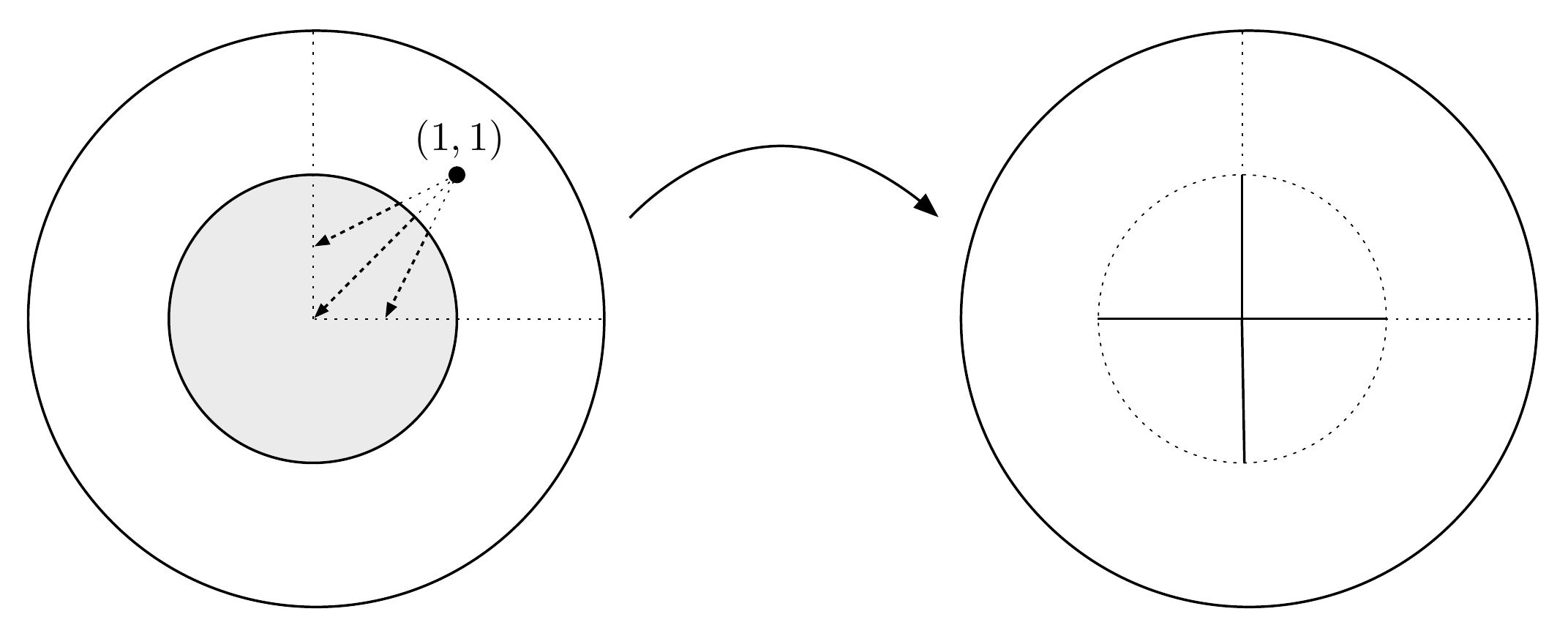}
\caption{A Lipschitz map  which is the identity on $\p B_2$ and which has $\frac 4 3 1_{\mb A(1,2)}$ as Jacobian. It maps $\mb S_1\cap\{x>0,y>0\}$ onto $\{0\}\times [0,1] \cup [0,1]\times\{0\}$ according to the dotted arrows.}
\label{fig:epsilonzero}
\end{figure}

In order to find an example where $f$ is bounded away from zero we need a substantially more intricate construction. Our goal is to prove the following result:

\begin{thm}
\label{thm:counterex}
For $\e\in [0,1]$, consider the family of data
\begin{equation}
\label{eq:deffepsilon}
f_\e\equiv \e 1_{B_1(0)} + 1_{\mb A(1,2)} + \frac{6-\e}{5} 1_{\mb A(2,3)}.
\end{equation}
There is a Lipschitz map $u_\e\colon B_3(0)\to B_3(0)$ such that
\begin{equation}
\label{eq:dirichletfepsilon}
\begin{cases}
\J u_\e = f_\e &  \text{in } B_3(0),\\
\J u_\e = \tp{id} &\text{on } \mb S_3,
\end{cases}
\end{equation}
and moreover there is a constant $C$, independent of $\e$, such that
\begin{equation}
\Vert \D u_\e \Vert_\infty \leq C.
\label{eq:lipschitzbound}
\end{equation}
\end{thm}

Let us just note that, once Theorem \ref{thm:counterex} is proved, the proof of Theorem \ref{bigthm:helein} is easily finished:

\begin{proof}[Proof of Theorem \ref{bigthm:helein}\ref{it:counterex}]
 Note that $f_\e\geq \e$ and that $\fint_{B_3(0)} f_\e \d x=1$, so that indeed $f_\e$ satisfies \eqref{eq:conditiondata}. Let $\phi_{\e}$ be the unique radial stretching solving \eqref{eq:dirichletfepsilon}, where $f_\e$ is as in \eqref{eq:deffepsilon}. Explicitly, $\phi_\e(z)=\rho_\e(r)\frac{z}{r}$ where, for $r\in (1,2)$,
\begin{equation}
\rho_\e (r) = \sqrt{r^2-1+\e} \quad \implies \quad |\rho'_\e(r)|^2 =\frac{r^2}{r^2-1+\e}.
\label{eq:radialblowsup}
\end{equation}
Using Lemma \ref{lemma:ball} we see that, as $\e\searrow 0$,
$$(9\pi)^{\frac{p-1}{2p}}\Vert \D\phi_\e \Vert_{L^{2p}(B_3)} \geq \Vert \D\phi_\e \Vert_{L^{2}(B_3)}\nearrow +\infty,$$
for any $p\in [1,\infty)$. Moreover, by \eqref{eq:lipschitzbound}, the maps $u_\e$ satisfy
$$\Vert \D u_\e \Vert_{L^{2p}(B_3)}\lesssim 1,$$
uniformly in $\e$ and $p$. This completes the proof.
\end{proof}
It thus remains to prove Theorem \ref{thm:counterex}. We begin by constructing an auxiliary map.
\begin{lemma}[Mapping a wedge onto an `A']\label{lem:wedgeA}
For $\e\in [0,1]$, consider the sets $$\Lambda\equiv \mb A_1(2,3)\cap\{y>0\},\qquad
A_\e\equiv \Lambda\cup\{1+\e(1-|x|)<y\leq2-|x|\}.$$ 
Let us write $\p \Lambda=\Gamma_1\cup \Gamma_2$, where $\Gamma_1\equiv \p \Lambda\exc A_\e$, and consider boundary data
$$\gamma_\e(x,y)= \begin{cases}
(x,y) & \text{on } \Gamma_1,\\
(x,1+\e(y-1)) & \text{on } \Gamma_2.
\end{cases}$$
There is a surjective Lipschitz map
$w_\e \colon \Lambda\to A_\e$, with $\Vert \D w_\e\Vert_\infty\leq C$, and such that
$$
\begin{cases}
\J w_\e = \frac{6-\e}{5} & \text{in } \Lambda,\\
w_\e=\gamma_\e &\text{on } \p \Lambda.
\end{cases}
$$
\end{lemma}

\begin{proof}
Take $\Lambda^+\equiv \Lambda\cap \{x>0\}$ and $A_\e^+\equiv A_\e\cap \{x>0\}$. Consider the map 
$\tau_\e\equiv (\tau_\e^1,\tau_\e^2)$ defined for $(x,y)\in \Lambda^+$ by
$$
\tau^1_\e (x,y)\equiv x,\qquad 
\tau^2_\e(x,y)\equiv 
\begin{cases}
 \frac{1}{2} (2 \e(x-1)  (x+y-3)-x^2+3x+y^2-y) & \tp{if } x\in [0,1],\\
\frac{1}{2} (x (2 y-5)+x^2+y^2-3 y +6) & \tp{if } x\in [1,2],\\
\frac{1}{2} y (x+y-1) & \tp{if }  x\in [2,3].
\end{cases}
$$
Since $\J \tau_\e = \p_y \tau^2_\e$, it follows that for $(x,y)\in \Lambda^+$we have
$$
\J \tau_\e (x,y)=
\begin{cases}
\e(x-1) +y-\frac{1}{2}& \tp{if } x\in [0,1],\\
x+y-\frac 3 2 & \tp{if } x\in [1,2],\\
\frac 1 2(x-1)+y & \tp{if }  x\in [2,3].
\end{cases}
$$ 
It is easy to check that $\J \tau_\e \in C^{0,1}(\overline{\Lambda^+})$ and $\J \tau_\e \geq \frac 1 2$ in $\Lambda^+$. 
Note that $\tau_\e\colon \Lambda^+\to A^+_\e$ is a bi-Lipschitz homeomorphism such that 
$$\tau_\e|_{\p \Lambda^+\cap \p \Lambda}=\gamma_\e|_{\p \Lambda^+\cap \p \Lambda}\qquad \tp{ and }
\qquad \tau_\e(\p \Lambda^+\exc \p \Lambda)= \p \Lambda^+\exc \p \Lambda
;$$
in fact, we found $\tau_\e$ by looking for maps with these properties such that $\tau_\e^2$ is a piecewise second order polynomial in $y$. See also Figure \ref{fig:wedges}.

\begin{figure}[htbp]
\centering
\includegraphics[width=0.7\textwidth]{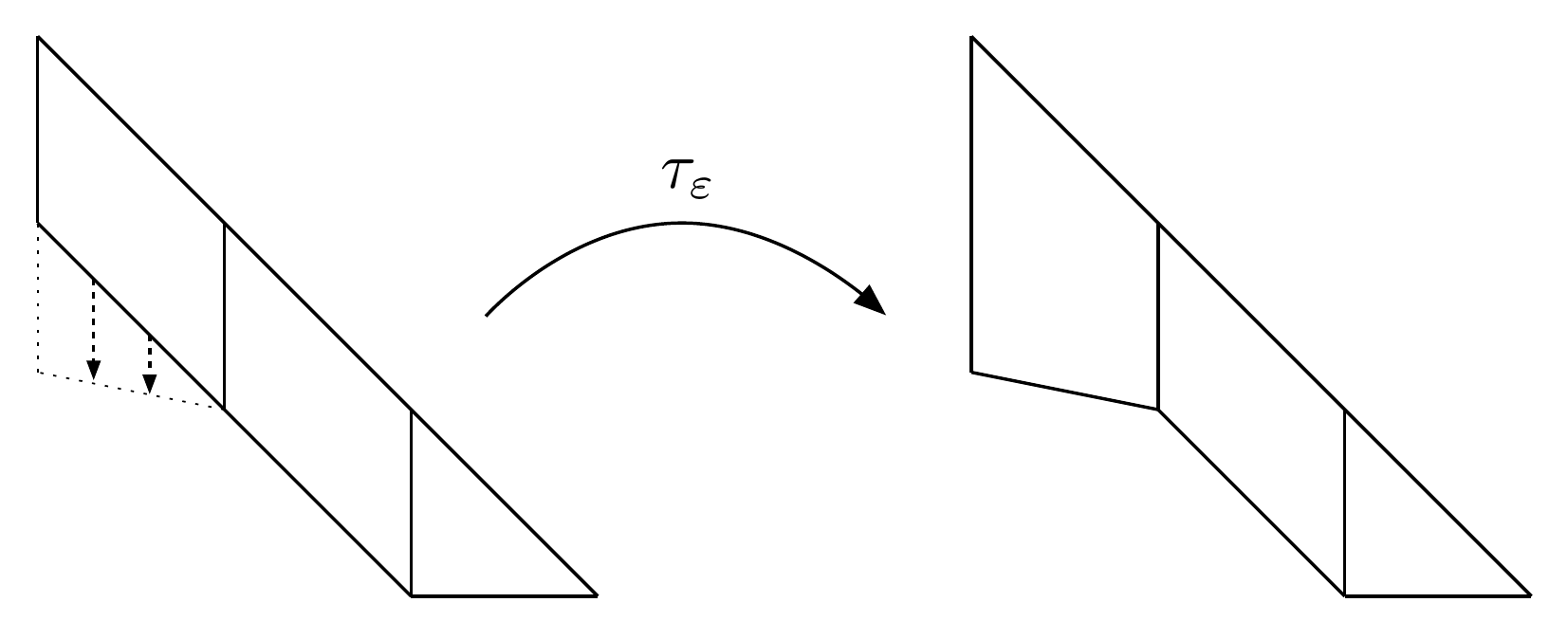}
\captionsetup{width=.8\linewidth}
\caption{The map $\tau_\e$, mapping $\Lambda^+$ onto $A^+_\e$. Apart from the segment with the two dashed arrows, $\tau_\e$ is the identity on $\p \Lambda^+$.}
\label{fig:wedges}
\end{figure}

We now want to apply the Dacorogna--Moser theory to find a map $w_\e\colon \Lambda^+\to A^+_\e$ with constant Jacobian. However, since $A^+_\e$ is just Lipschitz\footnote{The Dacorogna--Moser theory \cite{Dacorogna1990} requires the domain to be at least of class $C^{3,\alpha}$.} this cannot be done directly. Instead, we use Theorem \ref{thm:FonsecaParry} to find a bi-Lipschitz homeomorphism $a_\e\colon A^+_\e\to B_1(0)$ with constant Jacobian (explicitly, $\J a_\e =\frac{2\pi}{6-\e}$), and we take
a solution of
$$
\begin{cases}
\J \sigma_\e = g_\e
&\tp{in } B_1(0),\\
\sigma_\e = \tp{id} &\tp{on } \mb S^1,
\end{cases}
\qquad g_\e \equiv \frac{6-\e}{5} \frac{1}{\J \tau_\e\circ \tau_\e^{-1}\circ a_\e^{-1}}
$$
Note that, by the change of variables formula, and writing $\chi_\e \equiv a_\e\circ \tau_\e$,
\begin{align*}\int_{B_1(0)} g_\e =&\frac{6-\e}{5} \int_{B_1(0)} \frac{\J \chi_\e^{-1}}{\J \tau_\e \circ \chi_\e^{-1} \, \J \chi_\e^{-1}} = \frac{6-\e}{5}\int_{B_1(0)} \frac{\J \chi_\e \circ \chi_\e^{-1}}{\J \tau_\e \circ \chi_\e^{-1}} \J \chi_\e^{-1}\\
 =& \frac{6-\e}{5}\int_{\Lambda^+} \frac{\J \chi_\e}{\J \tau_\e} = \frac{6-\e}{5}\frac{2\pi}{6-\e} |\Lambda^+|=|B_1(0)|,\end{align*}
thus $g_\e$ satisfies the required compatibility condition.
For any $\a\in (0,1)$, we can additionally suppose that
$$
\Vert \sigma_\e -\tp{id}\Vert_{C^{1,\a}} \leq C\left(\a,\Vert g_\e\Vert_{C^{0,1}}\right) \Vert g_\e-1\Vert_{C^{0,\a}}\leq C(\a),
$$
see \cite[Theorem 8]{Riviere1996}. Here the last inequality follows from the fact that the bi-Lipschitz constants of $a_\e, \tau_\e$ are uniformly bounded with $\e\in [0,1]$, since the geometric parameters of $A_\e^+$, according to Definition \ref{def:classC}, are also bounded.
We now take $w_\e \colon \Lambda^+\to A_\e^+$ to be
$$w_\e \equiv a_\e^{-1} \circ \sigma_\e \circ a_\e \circ \tau_\e$$
and then extend $w_\e $ to $\Lambda\exc \Lambda^+$ through a reflection, similarly to \eqref{eq:reflections}. This yields the required map.
\end{proof}

\begin{figure}[htbp]
\centering
\includegraphics[width=0.84\textwidth]{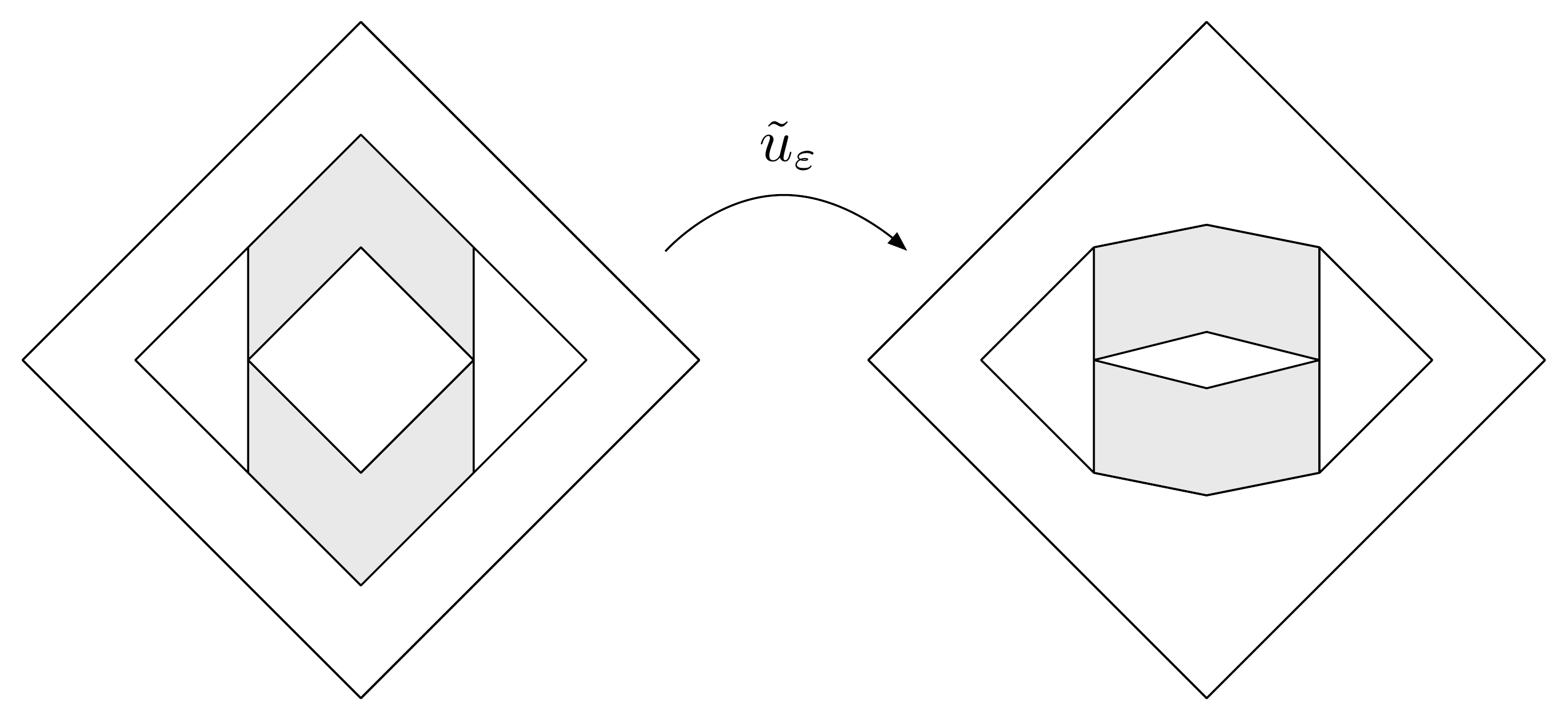}
\caption{The map $\tilde u_\e$ constructed in the proof of Theorem \ref{thm:counterex}.}
\label{fig:FinalMap}
\end{figure}

\begin{proof}[Proof of Theorem \ref{thm:counterex}]
Consider the map
$v_\e$ defined on $Q_2$ by
\begin{equation}
v_\e(x,y)=\begin{cases}
(x,\e y) & \tp{if }  (x,y)\in Q_1,\\
(x,y) & \tp{if }  (x,y)\in \mb A_1(1,2) \tp{ and } |x|>1,\\ 
(x,y-(1-\e)(1-|x|)) & \tp{if } (x,y)\in \mb A_1(1,2) \tp{ and } |x|<1, y>0,\\
(x,y+(1-\e)(1-|x|)) & \tp{if }  (x,y)\in \mb A_1(1,2)\tp{ and } |x|<1, y<0.
\end{cases}
\label{eq:Iwaniecsmap}
\end{equation}
It is easy to check that $\J v_\e=\e 1_{Q_1} + 1_{\mb A_{1}(1,2)}.$ Let $w_\e$ be the map from Lemma \ref{lem:wedgeA} and consider
$$\tilde u_\e \equiv \begin{cases}
v_\e & \tp{in } Q_2,\\
w_\e & \tp{in } \Lambda,\\
\bar w_\e & \tp{in } \bar \Lambda,
\end{cases}
\qquad \tp{where } \bar \Lambda \equiv \{(x,-y): (x,y)\in \Lambda\}
$$
and $\bar w_\e(x,y) \equiv (w_\e^1(x,-y),-w_\e^2(x,-y))$,
see Figure \ref{fig:FinalMap}. Thus
$$\J \tilde u_\e = \e 1_{Q_1}+1_{\mb A_1(1,2)} + \frac{6-\e}{5} 1_{\mb A_1(2,3)}.$$
Recall the map $\eta$ from Example \ref{ex:balltosquare} and let $R$ be a rotation by angle $\frac \pi 4$. Taking
$$u_\e \equiv (R \circ \eta)^{-1} \circ \tilde u_\e \circ (R\circ \eta),$$
the proof is finished.
\end{proof}

\section{Non-uniqueness of energy minimisers}
\label{sec:nonuniqueness}
The goal of this section is to prove Theorem \ref{bigthm:manyminimisers}, which we restate here:

\begin{thm}\label{thm:nonuniqueness}
Fix $1\leq p<\infty$. There is a radially symmetric function $f\in \mathscr H^p(\R^2)$ which has uncountably many $2p$-energy minimisers, modulo rotations.
\end{thm}

A more informative statement can be found in Corollary \ref{cor:formofminimisers}, at the end of the section. The proof of Theorem \ref{thm:nonuniqueness} relies mostly on elementary tools and the most sophisticated result that we use is the following:

\begin{thm}[Sierpi\'nski]\label{thm:sierpinski}
Let $(X_n)$ be  disjoint closed sets such that $I=\bigcup_{n\in \mb N} X_n$, where $I=[a,b]\subset \R$. There is at most one $n\in \N$ such that $X_n$ is non-empty.
\end{thm}

Theorem \ref{thm:sierpinski} is only needed to obtain uncountably many distinct minimisers, as non-uniqueness follows already from more elementary means. We also note that Theorem \ref{thm:sierpinski} holds more generally for a compact, connected Hausdorff space, see e.g.\ \cite[Theorem 6.1.27]{Engelking1989}. In the case of an interval there is a simple proof, which we give here for the sake of completeness:

\begin{proof}
Take $Y\equiv\bigcup_n \p X_n = I\exc \bigcup_n \tp{int}(X_n)$, which is closed, thus a complete metric space. 

We observe that the set $Y$ has empty interior in $I$, i.e.\ any open interval $L$ contains an open set $U$ disjoint from $Y$. 
Indeed, from the Baire Category Theorem we see that there is an open set $U\subseteq L$ and some $X_m$ which is dense in $U$. Since $X_m$ is closed, we must have $U\subseteq \tp{int}\,X_m$ and thus $U$ is disjoint from $Y$. 

By the Baire Category Theorem there is also some open subinterval $J$ of $I$ and some $n\in \N$ such that $\p X_n$ is dense in $Y\cap J$. Since $\p X_n$ is closed we have $\p X_n\cap J = Y\cap J$.  Thus $(Y\exc \p X_n)\cap J=\emptyset$.

Suppose now that $X_n\neq I$. It follows that $J$ intersects $Y\exc\p X_n$. Indeed, since $Y$ has empty interior in $I$, $J$ intersects $I\exc X_n$ and so it intersects $\tp{int}(X_k)$ for some $k$. Actually, $J$ must intersect $\p X_k$: otherwise, $\tp{int}(X_k)\cap J$ is non-empty, open and closed in $J$, thus $\tp{int}\, X_k=J$, since $J$ is connected; clearly this is impossible, since $X_k$ is disjoint from $X_n$. So we proved that $J$ intersects $Y\exc\p X_n$, contradicting the previous paragraph.
\end{proof}

We are now ready to begin the proof of Theorem \ref{thm:nonuniqueness}, whose core idea is contained in the following lemma.

\begin{lemma}\label{lemma:formofminimisers}
Let $u$ be a $2p$-energy minimiser for a radially symmetric function $f \in \mathscr{H}^p(\R^2)$.
For $\a_0\in [0,2\pi]$, consider the set
\begin{equation}
\label{eq:Xalpha}
X_{\alpha_0}\equiv \left\{\alpha \in [0,2\pi]:u_\a = u_{\a_0} \tp{ modulo rotations}\right\},
\qquad \tp{where } u_\a(z)\equiv u(e^{i \alpha} z).
\end{equation} 
Assume that $f\in C^0(B_R)$ has a sign. If $X_{\alpha_0}=[0,2\pi]$ then there is $k\in \Z\setminus\{0\}$ such that 
$$u(z)=\phi_k(z) \quad \tp{ in } B_R, \tp{ modulo rotations.}$$
\end{lemma}

\begin{proof}
If $X_{\a_0}=[0,2\pi]$ then, for any $\a\in [0,2\pi]$ and $z\in B_R$, we have
$|u(e^{i \a} z)|= |u(z)|$; that is, circles in $B_R$, centred at zero, are mapped to circles centred at zero.

For each $r\in (0,R)$, we have
$0\not \in u(\mb S_r)$. Indeed, for each ball $B\Subset B_R$, there is $c=c(B)>0$ such that $f\geq c$ in $B$ (or $f\leq -c$, but by reversing orientations we can always consider the first case without loss of generality). Thus, in $B_r$, $u$ is a map of integrable distortion and so it is both continuous and open \cite{Iwaniec1993a}.  Therefore $\partial(u(B_r))\subseteq u(\p B_r)=u(\mb S_r)$ and we see that $u(\mb S_r)\neq \{0\}$. Since $u(\mb S_r)$ is a circle, we conclude that $0\not\in u(\mb S_r)$.

By Proposition \ref{prop:lifting} we may write 
\begin{equation}\label{eq:polarrep}
u(r,\theta)=\psi(r,\theta)e^{i\gamma(r,\theta)}
\end{equation}
 where $\psi\in W^{1,2p}([0,R]\times [0,2\pi])$ and $\gamma\in W^{1,2p}([\e,R]\times [0,2\pi])$ satisfy \eqref{eq:compatibilitypolarcoords} and $\e>0$ is arbitrary.
For $r<R$, $u(\mb S_r)=\mb S_{r'}$, that is, $\psi(r,\theta)$ is independent of $\theta$. Thus, by \eqref{eq:jacobianpolarcoords}, $\J u =f$ reduces to
\begin{equation}
\p_r(\psi^2)\p_\theta \gamma= 2rf(r),
\label{eq:reducedjacobian}
\end{equation}
which is valid for almost every $(r,\theta)\in (0,R]\times[0,2\pi]$.
Since both $\psi$ and the right-hand side are independent of $\theta$ we must have
 $\gamma(r,\theta)=k\theta + \beta(r)$ and additionally there is the
compatibility constraint \eqref{eq:compatibilitypolarcoords} which yields $k\in \Z$. We may assume that $k\neq 0$: otherwise  \eqref{eq:reducedjacobian} shows that $f=0$ a.e., which is impossible. Since $u$ is a $2p$-energy minimiser, \eqref{eq:dirichletenergypolarcoords} readily implies that $\beta$ is constant. We integrate both sides of \eqref{eq:reducedjacobian}, using $\psi(0)=0$, to find
$$\psi(r)^2 =\frac 1 k \int_0^r 2s f(s) \d s \qquad \tp{ for } r<R.$$ 
Thus, modulo rotations, $u=\phi_k$ in $B_R$.
\end{proof}

In fact, the same argument applied in an annulus $\mb A(R_0,R)$ gives the following  variant:
\begin{lemma}\label{lemma:formOfMinimizersAnnulus}
Consider the setup of Lemma \ref{lemma:X0}, but replace $B_R$ by $\mb A(R_0,R)$.
Then there is $k\in \Z\setminus\{0\}$ and $c\in \R$ such that, in $\mb A(R_0,R)$,
$$u(z)= \psi(r)e^{2\pi i k \theta} \tp{ modulo rotations}, \qquad 
\tp{where }\psi(r)^2 = \frac 1 k\int_{R_0}^r 2 s f(s)\d s+c.$$
\end{lemma}

We now combine the previous two lemmas.

\begin{lemma}\label{lemma:X0}
There is a radially symmetric $f\in \mathscr H^p(\R^2)$, admitting a $2p$-energy minimiser $u$, for which we have $X_0\neq [0,2\pi]$, where $X_0$ is as in \eqref{eq:Xalpha}.
\end{lemma}

\begin{proof}
We take a function $f\colon \R^2\to \R$ satisfying the following conditions:
\begin{equation}
\label{eq:defdata}
\begin{split}
&f\in C^1(\R^2) \tp{ is radially symmetric},\\
&\int_{B_2} f\d x = \int_{\R^2} f \d x = 0\\
&f(r) <0 \tp{ if } 0<r<1, \quad  f(r)>0 \tp{ if } 1<r<2, \quad f(r)=(4-r)^+ \tp{ if } 3<r.
\end{split}
\end{equation}
By \cite[Theorem 4]{Kneuss2012}, there is $v\in C^1(\overline{B_4},\R^2)$ such that $\J v = f$ and $v=0$ on $\mb S_4$; in particular, by extending $v$ by zero outside $B_4$, we have $v\in W^{1,2p}(\R^2,\R^2)$. Since the $2p$-Dirichlet energy is convex, the Direct Method, combined with the sequential weak continuity of the Jacobian, shows that $f$ has at least one $2p$-energy minimiser and we call it $u$, using it to define the sets in \eqref{eq:Xalpha}. 

Suppose, for the sake of contradiction, that $X_0=[0,2\pi]$. Using Lemmas \ref{lemma:formofminimisers} and \ref{lemma:formOfMinimizersAnnulus}, we deduce that there are angles $\a,\a'\in [0,2\pi)$, numbers $k,k'\in \Z$ and $c\in \R$ such that
$$u=e^{i \a} \phi_k \tp{ in } B_1, \qquad u=e^{i \a'} \left(\psi(r) e^{2\pi i k' \theta}\right) \tp{ in } \mb A(1,2),$$
where, for $r\in (1,2)$,
$$\psi(r)^2 =\frac{1}{k'} \int_1^r 2 s f(s) \d s + c.$$
In the notation of Definition \ref{def:radialstretching}, we must have
\begin{align*}
e^{i\a+i k \theta} \frac{\rho(1)}{\sqrt{|k|}}\equiv\tp{Tr}_{\mb  S_1} u_{|B_1}=\tp{Tr}_{\mb  S_1} u_{|\mb A(1,2)} \equiv e^{i\alpha'+i k'\theta} c
\end{align*}
in $L^{2p}(\mb S_1)$. It is easy to conclude that $\alpha=\alpha'$, $k=k'$ and $c=\rho(1)/\sqrt{|k|}$, and so, modulo rotations, actually $u=\phi_k$ in $B_2$.
It is now easy to verify directly that, for $f$ as in \eqref{eq:defdata}, we have
$$\int_0^2 | \dot\rho(r)|^2 \d r = +\infty,$$
and so by Lemma \ref{lemma:ball} $u\not\in W^{1,2}(B_2,\R^2)$, which is a contradiction. Alternatively, one can infer that $u\not\in W^{1,2}(B_2,\R^2)$ from \cite[Theorem 3.4]{Lindberg2015}.
\end{proof}

\begin{proof}[Proof of Theorem \ref{thm:nonuniqueness}]
Let $f$ and $u$ be as in Lemma \ref{lemma:X0}. 
For each $\a\in [0,2\pi]$, it is easy to check that the set $X_\a$ is closed. We may write, for some index set $A$,
$$[0,2\pi]= \bigcup_{\a \in A} X_\alpha, \qquad \tp{where the union is disjoint.}$$
For distinct $\a,\a'\in A$, $X_\a$ and $X_{\a'}$ correspond to distinct equivalence classes of $2p$-energy minimisers, and so by Lemma \ref{lemma:X0} we must have $\# A>1$. But now Theorem \ref{thm:sierpinski} shows that $A$ must be uncountable.
\end{proof}

We also note that the proof of Lemma \ref{lemma:formofminimisers} yields the following corollary:

\begin{cor}\label{cor:formofminimisers}
Let $f\in \mathscr H^p(\R^2)$ be radially symmetric and suppose $u$ is its unique $2p$-energy minimiser, modulo rotations. If $u$ is continuous then $u=\phi_k$ for some $k \in \Z \setminus \{0\}$.
\end{cor}

Clearly the continuity assumption is not restrictive if $p>1$.

\begin{proof}
As in the proof of Lemma \ref{lemma:formofminimisers} we conclude that $u$ maps circles centred at zero to circles and that $(r,\theta)\mapsto |u(r e^{i \theta})|$ is independent of $\theta$. Thus we write simply $|u(r)|$.

We show that the set $\{r \in (0,\infty) \colon |u(r)| > 0\}$ is connected. Suppose, by way of contradiction, that there are $r_1 < r_2 < r_3$ such that $|u(r_1)|, |u(r_3)| > 0$ but $|u(r_2)| = 0$. We get another $2p$-energy minimiser for $f$ by setting
\[v(z) = \begin{cases}
           u(z), & |z| \le r_2, \\
           e^{i \pi} u(z), & |z| > r_2,
         \end{cases}\]
contradicting the assumption that the $2p$-energy minimiser for $f$ is unique modulo rotations.

Thus we can write, for some $0\leq R_1 \leq R_2\leq \infty$,
$$\left\{r \in (0,\infty) \colon |u(r)| > 0\right\}=(R_1,R_2).$$
We can use Lemma \ref{lemma:formOfMinimizersAnnulus} to conclude that $u=\phi_k$ in $\mb A(R_1,R_2)$, modulo rotations. Moreover, clearly we must have $f(r)=0$ if $r\not \in (R_1,R_2)$. Thus $\phi_k(z)=0$ if $r\not\in (R_1,R_2)$ and so $u=\phi_k$ outside $\mb A(R_1,R_2)$ as well.
\end{proof}

\bigskip
{\small \noindent \textbf{Acknowledgments.}
We thank Tadeusz Iwaniec for suggesting the formula of the map in \eqref{eq:Iwaniecsmap}.
}

\let\oldthebibliography\thebibliography
\let\endoldthebibliography\endthebibliography
\renewenvironment{thebibliography}[1]{
  \begin{oldthebibliography}{#1}
    \setlength{\itemsep}{1.5pt}
    \setlength{\parskip}{1.5pt}
}
{
  \end{oldthebibliography}
}

{\small
\bibliographystyle{acm}
\bibliography{/Users/antonialopes/Dropbox/Oxford/Bibtex/library.bib}
}

\end{document}